\documentclass{amsart}
\usepackage{adjustbox,lipsum}
\usepackage{amsmath}
\usepackage{amssymb}
\usepackage{MnSymbol}
\usepackage[shortlabels]{enumitem}
\usepackage{tabu}
\usepackage{tikz-cd}
\usepackage{extarrows}
\usepackage{makecell}

\usepackage[
        colorlinks, citecolor=darkgreen,
        backref,
        pdfauthor={Samir Siksek},
]{hyperref}

\usepackage{comment}

\newcommand{\F}{\mathbb{F}}

\newcommand{\Q}{\mathbb{Q}}
\newcommand{\R}{\mathbb{R}}
\newcommand{\Z}{\mathbb{Z}}

\newcommand{\ff}{\mathfrak{f}}
\newcommand{\fq}{\mathfrak{q}}
\newcommand{\fl}{\mathfrak{l}}
\newcommand{\fb}{\mathfrak{b}}
\newcommand{\fa}{\mathfrak{a}}

\newcommand{\cN}{\mathcal{N}}

\newcommand{\cS}{\mathcal{S}}
\newcommand{\cT}{\mathcal{T}}
\newcommand{\OO}{\mathcal{O}}

\newcommand{\fN}{\mathfrak{N}}

\newcommand{\fp}{\mathfrak{p}}

\DeclareMathOperator{\Frob}{Frob}

\DeclareMathOperator{\Cl}{Cl}

\DeclareMathOperator{\Gal}{Gal}

\DeclareMathOperator{\Norm}{Norm}

\DeclareMathOperator{\ord}{ord}

\DeclareMathOperator{\Tr}{Tr}

\renewcommand{\setminus}{-}

\newcommand{\GL}{\operatorname{GL}}

\newtheorem{theorem}{Theorem}
\newtheorem{lemma}[theorem]{Lemma}
\newtheorem{conj}[theorem]{Conjecture}

\newtheorem{proposition}[theorem]{Proposition}

\theoremstyle{definition}

\theoremstyle{remark}
\newtheorem{remark}[equation]{Remark}

\definecolor{darkgreen}{rgb}{0,0.5,0}

\begin{document}

\title[]{
Perfect powers in elliptic divisibility sequences
}

\begin{abstract}
Let $E/\Q$ be an elliptic curve given by
an integral Weierstrass equation. Let $P \in E(\Q)$ be a point of infinite order, 
and let $(B_n)_{n\ge 1}$ be the 
elliptic divisibility sequence generated by $P$.
This paper is concerned with a question posed in
2007 by Everest, Reynolds and Stevens: does $(B_n)_{n \ge 1}$
contain only finitely many perfect powers?
We answer this question positively under the following
three additional assumptions:
\begin{enumerate}[(i)]
	\item $P$ is non-integral;
	\item $\Delta(E) >0$, where $\Delta(E)$ is the discriminant
		of $E$;
	\item $P \in E_0(\R)$, where $E_0(\R)$ denotes the connected
		component of identity.
\end{enumerate}
Our method makes use of 
Galois representations
attached to elliptic curves defined over totally real fields,
and their modularity.

We can deduce the same theorem without assumptions
(ii) and (iii) provided we assume some standard
conjectures from the Langlands programme.
\end{abstract}

\author{Maryam Nowroozi}
\address{Mathematics Institute\\
   University of Warwick\\
    CV4 7AL \\
    United Kingdom}

\email{maryam.nowroozi@warwick.ac.uk}

\author{Samir Siksek}

\address{Mathematics Institute\\
    University of Warwick\\
    CV4 7AL \\
    United Kingdom}

\email{s.siksek@warwick.ac.uk}

\date{\today}
\thanks{
Nowroozi is supported by the EPSRC studentship.
Siksek is supported by the
EPSRC grant \emph{Moduli of Elliptic curves and Classical Diophantine Problems}
(EP/S031537/1). }

\makeatletter
\@namedef{subjclassname@2020}{%
  \textup{2020} Mathematics Subject Classification}
\makeatother


\maketitle

\section{Introduction}
The problem of studying perfect powers in certain sequences has long been
of interest to number theorists. We mention two
well-known results.
\begin{itemize}
\item Peth\H{o} \cite{Petho} and, independently,
Shorey and Stewart \cite{ShoreyStewart} showed that
there are only finitely many perfect powers in any non-trivial
binary recurrence sequence.
\item Bugeaud, Mignotte, and Siksek \cite{FibLuc}
proved that the only perfect powers in the Fibonacci sequence
are $0$, $1$, $8$ and $144$. 
\end{itemize}

This paper is concerned with perfect powers in elliptic divisibility
sequences, which we now introduce.
Let $E/\mathbb{Q}$ be an elliptic curve given by an
integral Weierstrass model 
\begin{equation}\label{eqn:model}
    y^2+a_1xy+a_3y=x^3+a_2x^2+a_4x+a_6, \ \ a_1,\ldots,a_6\in \mathbb{Z}
\end{equation}
and let $P\in E(\mathbb{Q})$ be a point of infinite order. For $n\geq 1$,
we may write the multiple $nP$ as
\begin{equation}
    nP \; =\; \left(\frac{A_n}{B_n^2}, \frac{C_n}{B_n^3}\right)
\end{equation}
where $A_n$, $B_n$, $C_n \in \Z$ with $B_n \ge 1$
	and $\gcd(B_n,A_n C_n)=1$.
The sequence $\mathcal{B}_{E,P}=(B_n)_{n\geq1}$ is called the \emph{elliptic divisibility sequence} associated to $E$ and $P$.
It is a divisibility sequence in the sense that if $m\mid n$, then $B_m \mid B_n$.
We say that a prime $p$ is a \textbf{primitive divisor} of the term $B_n$ if $p \mid B_n$
but $p \nmid B_m$ for all $m<n$.
	The following theorem of Silverman \cite[Lemma 9]{SILVERMANABC} is the most famous result regarding
the arithmetic of elliptic divisibility sequences.
\begin{theorem}[Silverman]\label{thm:silverman}
Let $(B_n)_{n \ge 1}$ be an elliptic divisibility sequence.
Then, at most finitely many terms in the sequence do not have a primitive divisor.
\end{theorem}

In 2007, Everest, Reynolds and Stevens \cite{Everest} posed the 
question of whether an elliptic divisibility sequence $(B_n)_{n \ge 1}$
contains only finitely many perfect powers.
They proved the following somewhat weaker result.
\begin{theorem}[Everest, Reynolds and Stevens] \label{thm:ERS}
Let $(B_n)_{n \geq 1}$ be an elliptic divisibility sequence.
For any fixed $f \ge 2$, there are only finitely many $n$ for which $B_n$ is an $f^{\text{th}}$ power.
\end{theorem}
The proof of Theorem~\ref{thm:ERS} crucially relies
on Faltings' theorem (the Mordell conjecture), and is therefore ineffective.
Theorem~\ref{thm:ERS} has been extended to the 
function field setting in \cite{CornelissenReynolds}
and to products of terms in elliptic divisibility sequences
\cite{Hajdu}.
However, the question of Everest, Reynolds and Stevens has remained open. The
only partial results we are aware of are due to Reynolds \cite{Reynolds}
and to Alfaraj \cite{Alfaraj2}. 

Reynolds proved that there are finitely many perfect powers in an elliptic
divisibility sequence  
arising from a non-integral point on an elliptic curve of the form
$y^2=x^3+d$, where $d \ne 0$ is an integer.
Alfaraj proves that there are only finitely many perfect powers
in an elliptic divisibility sequence arising from a non-integral
point on elliptic curves of the form $y^2=x(x^2+b)$,
where $b$ is a positive integer.
Both proofs attach a suitable Frey curve to the
equation $B_n=u^\ell$ and make use of the modularity
of the mod $\ell$ Galois representation of the Frey curve.
In the work of Reynolds, the Frey curve is defined over 
the rationals, and in Alfaraj's paper the Frey curve is defined
over a quadratic field.

The purpose of this paper is to generalize the results of Reynolds and Alfaraj
to general Weierstrass models. More precisely, we prove the following theorem.
\begin{theorem}\label{thm:Main1}
Let $E/\Q$ be an elliptic curve given by an integral 
Weierstrass equation \eqref{eqn:model}. Let $P \in E(\Q)$
and let $(B_n)_{n \ge 1}$ be the corresponding 
elliptic divisibility sequence. Suppose the following three
conditions hold:
\begin{enumerate}[(i)]
	\item $P$ is non-integral (equivalently $B_1 >1$);
	\item $\Delta(E)>0$ where $\Delta(E)$ is the discriminant
		of $E$;
	\item $P \in E_0(\R)$, where $E_0(\R)$ denotes the connected
		component of the identity.
\end{enumerate}
Then there are only finitely many perfect powers in the sequence
$(B_n)_{n \ge 1}$.
\end{theorem}
\noindent \textbf{Remark.} 
Suppose (i) and (ii) hold, but not (iii). The assumption (ii)
implies $[E(\R): E_0(\R)]=2$. Thus $2P \in E_0(\R)$.
It follows from the theorem that the 
subsequence $(B_{2n})_{n \ge 1}$ contains only finitely many
perfect powers.

\medskip

A key innovation of our paper is 
the construction of a Frey curve
attached to the equation $B_n=u^\ell$ where
$B_n$ arises from a point of infinite order 
on a general Weierstrass model.
This Frey curve is defined over a number
field $L$ of degree at most $24$, and assumptions (ii) and (iii)
are needed to ensure that this field is totally real,
whence we may use known results on the modularity
of elliptic curves over totally real fields.
We can make do without assumptions (ii) and (iii)
if we assume the three standard conjectures,
which we state in Section~\ref{sec:conjectures},
but for now briefly indicate what they are:
\begin{enumerate}[(I)]
\item Serre's modularity conjecture
over number fields (Conjecture~\ref{conj:modularity});
\item Serre's uniformity conjecture for elliptic curves over number fields
(Conjecture~\ref{conj:uniformity});
\item the Ramanujan--Petersson conjecture for eigenforms over number fields
(Conjecture~\ref{conj:RP}).
\end{enumerate}

\begin{theorem}\label{thm:Main2}
Let $E/\Q$ be an elliptic curve given by an integral 
Weierstrass equation \eqref{eqn:model}. Let $P \in E(\Q)$
and let $(B_n)_{n \ge 1}$ be the corresponding 
elliptic divisibility sequence. Suppose $P$
is non-integral (equivalently $B_1>1$).
Assume Conjectures~\ref{conj:modularity}, \ref{conj:uniformity}
and \ref{conj:RP}.
Then there are only finitely many perfect powers in the sequence
$(B_n)_{n \ge 1}$.
\end{theorem}

The paper is organised as follows. In Section~\ref{sec:background}
we recall standard properties of elliptic divisibility
sequences, and we generalise
a result of Alfaraj that is in turn a consequence
of Silverman's primitive divisor theorem. In Section~\ref{sec:prelim}
we state three theorems concerning Galois representations
attached to elliptic curves over totally real fields.
These theorems play the role, in our setting, that is played
by the theorems of Mazur, Ribet and Wiles in the proof of Fermat's 
Last Theorem. In Section~\ref{sec:key} we state a proposition
(Proposition~\ref{prop:prelim})
that asserts the existence of a suitable Frey curve,
defined over a number field $L$ of degree at most $24$,
attached to a solution of $B_n=u^\ell$;
hypotheses (ii) and (iii) of Theorem~\ref{thm:Main1}
ensure that $L$ is totally real. We also show
that Theorem~\ref{thm:Main1} follows from Proposition~\ref{prop:prelim},
using the three theorems. Proposition~\ref{prop:prelim}
is proved in Section~\ref{sec:prop}. Finally, in Section~\ref{sec:conjectures}
we give precise statements of the three conjectures
above, and we explain how they allow us, together with
relatively minor modifications of earlier arguments,
to deduce Theorem~\ref{thm:Main2}.


\section{Background}
\label{sec:background}

The proposition below recalls a standard property of elliptic divisibility sequences. 
\begin{proposition}{\label{divprop}}
Let $( B_n)_{n \ge 1}$ be an elliptic divisibility sequence generated by a non-torsion point on 
an elliptic curve $E/\Q$ given by an integral Weierstrass model.
\begin{enumerate}[(i)]
\item 
Let $p$ be a prime, and $n$, $m \in \mathbb{N}$. Suppose $v_p(B_n)>0$. Then
\begin{equation}\label{eqn:val}
    v_p(B_{nm})=v_p(B_n)+v_p(m)+ O(1)
\end{equation}
 where the $O(1)$ constant does not depend on $m$ and $n$.
\item $(B_n)_{n \ge 1}$ is a strong divisibility sequence, i.e., for all $n$, $m \in \mathbb{N}$,
\begin{equation}\label{eqn:val2}
 \gcd(B_m,B_n)=B_{\gcd(m,n)}.
\end{equation}
\end{enumerate}
\begin{proof}
These are consequences
of the theory of formal groups of elliptic curves.
For proofs, see Lemma IV.4.3 and Lemma IV.3.1 of \cite{Streng}.
\end{proof}
\end{proposition}
We point out \cite[Corollary 4.5]{Streng} that \eqref{eqn:val} can be simplified
to $v_p(B_{nm})=v_p(B_n)+v_p(m)$ if $p \ne 2$, or 
if $p=2$ but $a_1$ is even, where $a_1$ is the coefficient of $xy$ in the Weierstrass model
\eqref{eqn:model}.

The following lemma is a crucial step in proving Theorem \ref{thm:Main1}. 
It is a generalisation of Alfaraj's \cite[Lemma~3.4]{Alfaraj2}
which copes with the case where $a_1$ is odd.

\begin{lemma}\label{lem:prime}
Let $(B_n)_{n \ge 1}$ be an elliptic divisibility sequence generated by a
non-torsion point on some elliptic curve over $\Q$ given by an
integral
Weierstrass model. Suppose that $B_1>1$.
Let $T$ be a finite set of primes. There exists a positive integer $\kappa$
and a prime $p\notin T$ such that if $B_n$ is an $\ell$-th power for some 
$\ell>\kappa$,
then $p\mid B_n$.  
\end{lemma}
\begin{proof}
Let $q \mid B_1$ be a prime.
Suppose $B_n$ is an $\ell$-th power.
By \eqref{eqn:val2}, we have $q\mid B_n$, and hence $\ell \leq v_q(B_n)$.
From \eqref{eqn:val}, we have
\[
	\ell \; \leq  \; v_q(B_1)+v_q(n)+r,
\]
for some integer $r$, depending only on the elliptic
divisibility sequence.  Hence, 
\begin{equation}\label{eqn:ndiv}
	v_q(n) \; \geq \; \ell -v_q(B_1)-r.
\end{equation}
By Silvermans's primitive divisor theorem (Theorem \ref{thm:silverman}), 
there exists $\kappa^\prime$ such that if $m\geq \kappa^\prime$, 
then $B_m$ is divisible
by some prime $p \notin T$.  Let $\kappa$ be an integer such that
that $\kappa>v_q(B_1)+r$ and 
$q^{\kappa-v_q(B_1)-r}>\kappa^\prime$. 
Then, there is some prime $p \notin T$ such that 
\[
		p\mid B_{q^{\kappa-v_q(B_1)-r}}.
\]
If $\ell > \kappa$, then
\[
	q^{\kappa-v_q(B_1)-r}\mid q^{\ell-v_q(B_1)-r}\mid n
\]
by \eqref{eqn:ndiv}.
So, 
\[
	p\mid B_{q^{\kappa-v_q(B_1)-r}}\mid B_n.
\]
\end{proof}

\section{Preliminaries from Modularity and Galois Representations of Elliptic Curves}\label{sec:prelim}
We prove Theorem~\ref{thm:Main1} by attaching a Frey elliptic curve, defined over a totally
real field, to a term of the elliptic divisibility sequence that is a perfect power,
and applying the so-called \lq\lq modular method\rq\rq\ that is inspired by the 
Wiles' proof of Fermat's Last Theorem;
introductions to the method can be found in \cite{EkinSurvey} and \cite{SiksekModular}. 
The proof of Fermat's Last Theorem
rests on three major pillars that we now recall.
\begin{itemize}
	\item Modularity of semistable elliptic curves over the rationals due to Wiles \cite{Wiles}.
\item The irreducibility of Galois representations attached
to the $\ell$-torsion of elliptic curves over the rationals, 
for sufficiently large primes $\ell$,
due to Mazur \cite{Mazur78}.
\item Ribet's Level Lowering Theorem \cite{Ribet-1990}.		
\end{itemize}
We shall need (somewhat weak) generalisations of such results to the context
of elliptic curves over totally real fields.
In particular, we need the following theorem proved in \cite[Theorem 5]{SamiretalModular}.
\begin{theorem}[Freitas, Le Hung and Siksek]\label{thm:modularity}
Let $L$ be a totally real field. There are at most
finitely many $\overline{L}$-isomorphism
classes of elliptic curves over $L$ that are non-modular.
\end{theorem}

The following theorem, stated in \cite[Theorem 7]{SamirSummary}, is a summary of level lowering theorems
due to Fujiwara, Jarvis and Rajaei, specialised to the context of
modular elliptic curves.
\begin{theorem}\label{thm:ll}
Let $L$ be a totally real number field, 
and $F/L$ an elliptic curve of conductor $\cN$. Let $\ell$ 
be a rational prime, and write $\overline{\rho}_{F,\ell}$
for the mod $\ell$ representation attached to $F$. 
For a prime ideal $\fq$ of $\OO_L$,
let $\Delta_{\fq}$ be
the discriminant of a local minimal model for $F$ at $\fq$. Let
\begin{equation}\label{eqn:Nl}
	\mathcal{M}:=\prod_{\substack{\fq \mid\mid\cN\\ \ell \mid v_{\fq}(\Delta_{\fq})}} \fq,
\qquad
\fN:=\frac{\cN}{\mathcal{M}}.
\end{equation}
Suppose the following
\begin{enumerate}
    \item[(i)] $\ell \geq 5$, the ramification index $e(\mathfrak{l}/\ell)<\ell-1$ for all $\mathfrak{l}\mid \ell$, and $\mathbb{Q}(\zeta_\ell)^{+}\not\subseteq L$;
    \item[(ii)]  $F$ is modular;
    \item[(iii)] $\overline{\rho}_{F,\ell}$ is irreducible;
    \item[(iv)]  $F$ is semistable at all $\mathfrak{l}\mid \ell$;
    \item[(v)] $\ell\mid v_{\mathfrak{l}}(\Delta_{\mathfrak{l}})$ for all $\mathfrak{l}\mid \ell$.
\end{enumerate}
Then, there is a Hilbert eigenform $\ff$ of parallel weight $2$ 
that is new at level $\fN$ and some prime ideal $\lambda$ of the ring of integers of $\mathbb{Q}_{\ff}$ such that $\lambda\mid \ell$ and
$\overline{\rho}_{F,\ell}\sim \overline{\rho}_{\ff,\lambda}$.
\end{theorem}

\begin{remark}{\label{remark}}
Following the notation of Theorem \ref{thm:ll}, suppose that
$\overline{\rho}_{F,\ell}\sim \overline{\rho}_{\ff,\lambda}$. Let
$\fq$ be a prime ideal of $\OO_L$.  
By comparing the
traces of the images of Frobenius at $\fq$ in
$\overline{\rho}_{F,\ell}$ and $\overline{\rho}_{\ff,\lambda}$ we have the
following.
\begin{enumerate}[(a)]
\item If $\fq \nmid \ell\cN$, then $a_{\fq}(F)\equiv a_{\fq}(\ff) \pmod{\lambda}$.
\item If $\fq \nmid \ell\fN$, and $\fq \mid\mid \cN$, then $\pm(\Norm_{L/\mathbb{Q}}(\fq)+1)\equiv  a_{\fq}(\ff) \pmod{\lambda}$
\end{enumerate}
\end{remark}

The following result by Freitas and Siksek in \cite{siksekbound}, shows that,
under certain assumptions, hypothesis (iii) in Theorem \ref{thm:ll} 
is satisfied
whenever $\ell$ is sufficiently large.
This theorem makes essential use of Merel's celebrated
uniform boundedness theorem \cite{Merel}.
\begin{theorem}\label{thm:irred}
Let $L$ be a Galois totally real field. There is an effective constant $C_L$,
depending only on $L$, such that the following holds. If $\ell> C_L$ is
prime, and $F$ is an elliptic curve over $L$
which is semistable at all $\mathfrak{l}\mid \ell$, then 
$\overline{\rho}_{F,\ell}$ 
is irreducible.
\end{theorem}


\section{A Key Proposition and the Proof of Theorem \ref{thm:Main1}}
\label{sec:key}

In this section, we state a key result towards
the proof of Theorem~\ref{thm:Main1}. This is Proposition~\ref{prop:prelim}
below, which in essence asserts the existence
of a suitable Frey curve for the perfect powers problem
for elliptic divisibility sequences.
We also show that this proposition implies Theorem~\ref{thm:Main1},
making use of the theorems in the previous section.
We complete the proof of Theorem~\ref{thm:Main1} in 
Section~\ref{sec:prop} by proving Proposition~\ref{prop:prelim}.

\medskip

Let $E$ be an elliptic curve given by an integral
Weierstrass model \eqref{eqn:model}. Let $P \in E(\Q)$.
We suppose that $E$ and $P$ satisfy hypotheses
(i), (ii) and (iii) of Theorem~\ref{thm:Main1}.
Let $(B_n)_{n \ge 1}$
be elliptic divisibility sequence corresponding to $P$.
Thus we may write 
\begin{equation}\label{eqn:divseq}
nP=(A_n/B_n^2,C_n/B_n^3), 
\end{equation}
where
\begin{equation}\label{eqn:cond}
A_n,B_n,C_n \in \Z, \qquad
B_n \ge 1, \qquad 
\gcd(A_n C_n,B_n)=1.
\end{equation}
Moreover, as $P$ is non-integral, we have that $B_1>1$.

We shall find it convenient to work with a \lq\lq short Weierstrass model\rq\rq\
which we denote by $E^\prime$,
\begin{equation}\label{eqn:Eprime}
	E^\prime \; : \; {y^\prime}^2 = {x^\prime}^3+a {x^\prime}^2 +b {x^\prime} +c
\end{equation}
where
\begin{equation}\label{eqn:sub1}
a=4a_2+a_1^2,\quad b=8(2a_4+a_1 a_3), \quad c=16(a_3^2+4 a_6).
\end{equation}
We observe that the two models are related by the transformation,
\begin{equation}\label{eqn:sub2}
	x^\prime=4x, \qquad y^\prime=4(2y+a_1x+a_3), 
\end{equation}
and
we shall write $\theta_1$, $\theta_2$, $\theta_3$ for the roots
of the polynomial ${x^\prime}^3+a {x^\prime}^2 +b {x^\prime} +c$.
We let $P^\prime$ be the point corresponding to $P$
on the model $E^\prime$.
Let
\begin{equation}\label{eqn:KL}
	K=\Q(\theta_1,\theta_2,\theta_3), \qquad
	L=K\left(
	\sqrt{x^\prime(P^\prime)- \theta_1},
	\sqrt{x^\prime(P^\prime) - \theta_2},
	\sqrt{x^\prime(P^\prime)- \theta_3}\right).
\end{equation}
It is clear that the number fields $K$ and $L$
depend only on $E$ and $P$. Clearly $K$ has 
degree at most $6$. Moreover, $[L:K]$ is at most $4$,
since since the product of the three $x^\prime(P^\prime)-\theta_i$
is $y^\prime(P^\prime)^2$. Hence $L$ has degree at most $24$.
Although we shall not need this later, 
we note in passing that $L$ is the extension over which the 
four points satisfying $2Q=P^\prime$ are defined;
this is clear from the injectivity of the $2$-descent
map $\phi$ in the proof of Lemma~\ref{lem:2descent}.
\begin{lemma}
The number fields $K$ and $L$ are totally real and Galois.
\end{lemma}
\begin{proof}
It is clear that $K$ and $L$ are Galois.
Hypothesis (ii) of Theorem~\ref{thm:Main1}
implies that $\theta_1$, $\theta_2$, $\theta_3$
are real, and hypothesis (iii) implies
that $x^\prime(P^\prime) \ge \theta_1$, $\theta_2$, $\theta_3$.
Thus $K$ and $L$ are totally real.
\end{proof}

We consider the equation
\begin{equation}\label{eqn:Bnuell}
	B_n=u^\ell
\end{equation}
where $u$ is a positive integer and $\ell$ is a prime.
The proof of Theorem~\ref{thm:Main1}
crucially relies on the following intermediate result.
\begin{proposition}\label{prop:prelim}
There is a positive constant $\kappa$,
a finite set of prime ideals $\cS$ of $\OO_L$, 
and a prime ideal $\fp \notin \cS$,
all depending only on $E$ and $P$,
such that the following holds.
Let $(n,u,\ell)$ be a solution to \eqref{eqn:Bnuell}
with $\ell>\kappa$. Then there is an elliptic curve $F/L$,
given by an integral Weierstrass model with discriminant $\Delta_F$, having the following properties.
\begin{enumerate}[(a)]
\item For all prime ideals $\fq \notin \cS$,
the model $F$ is minimal and semistable at $\fq$, 
and $\ell \mid v_{\fq}(\Delta_F)$.
\item The elliptic curve $F$ has multiplicative reduction
at $\fp$ and $\ell \mid v_{\fp}(\Delta_F)$.
\end{enumerate}
\end{proposition}

We prove Proposition~\ref{prop:prelim} in Section~\ref{sec:prop}.
We now show that the proposition implies Theorem~\ref{thm:Main1}.
\begin{proof}[Proof of Theorem~\ref{thm:Main1}]
We claim that the prime exponent $\ell$ in \eqref{eqn:Bnuell} is 
bounded. Theorem~\ref{thm:Main1} follows from our claim
thanks to the theorem of Everest, Reynolds and Stevens
(Theorem~\ref{thm:ERS}).

To prove our claim, let $F$, $\kappa$, $\cS$, $\fp$
be as in Proposition~\ref{prop:prelim}.	
By Theorem~\ref{thm:modularity}, there is at most a finite
number of $j$-invariants $j_1,j_2,\dotsc,j_m$ of non-modular
elliptic curves defined over $L$. 
We let 
\[
	\kappa^\prime=\max(4,\lvert \Delta_L \rvert, C_L, 
	\Norm_{L/\Q}(\fp),\kappa,\kappa_1,\kappa_2)
\]
where
$\Delta_L$ is the discriminant of $L$, 
the constant $C_L$ is as in Theorem~\ref{thm:irred},
and
	\[
\kappa_1=\max(-v_{\fp}(j_i) \; : \; 1 \le i \le m),
	\qquad
\kappa_2 =\max(\Norm_{L/\Q}(\fq) \; : \; \fq \in \cS\}).
\]
It is clear that $\kappa^\prime$ 
depends only on $E$ and $P$.

Let $\ell>\kappa^\prime$ be a prime and let $n$
and $u$ be positive integers such that \eqref{eqn:Bnuell}
holds. Let $F/L$ be the elliptic curve, given by an integral model,
whose existence is asserted in Proposition~\ref{prop:prelim}.
Write $j(F)$
for the $j$-invariant of $F$ and $\cN$ for its conductor. 
By (b), $F$ has multiplicative
reduction at $\fp$ and $\ell \mid v_{\fp}(\Delta_F)$.
Thus, for $1 \le i \le m$,
\[
	-v_{\fp}(j_i) \; \le \; \kappa_1 \; < \;	\ell \; \le \; 
	v_{\fp}(\Delta_F) \; = \; -v_{\fp}(j(F)).
\]
It follows that $j(F) \ne j_i$ for $1 \le i \le m$ and so 
the elliptic curve $F$ is modular.

Since $\ell>\kappa_2$,
all $\mathfrak{l} \mid \ell$ satisfy $\mathfrak{l} \notin S$,
and hence, by (a), the elliptic curve $F$ is semistable at all $\mathfrak{l} \mid \ell$.
Hence by Theorem~\ref{thm:irred}, as $\ell>C_L$, the mod $\ell$ representation $\overline{\rho}_{F,\ell}$ is irreducible.

Next, we apply Theorem~\ref{thm:ll}.
We have that $\ell \ge 5$. Moreover, $\ell > \lvert \Delta_L\rvert$
therefore $\ell$ is unramified in $L$. 
As $\Q(\zeta_\ell)^+$ is ramified at $\ell$
we conclude that $\Q(\zeta_\ell)^+ \not \subseteq L$.
Hence hypothesis (i)
of Theorem~\ref{thm:ll} is satisfied. We have already checked that hypotheses
(ii), (iii) and (iv) hold. Hypothesis (v) follows from (a)
as all $\mathfrak{l} \mid \ell$ satisfy $\mathfrak{l} \notin S$.
Thus by Theorem~\ref{thm:ll} there is a Hilbert newform
$\ff$ of parallel weight $2$ and level $\fN$
such that $\overline{\rho}_{F,\ell} \sim \overline{\rho}_{\ff,\lambda}$.
It follows from (a) and the recipe \eqref{eqn:Nl} that $\fN$
is supported on the primes belonging to $\cS$.
Moreover, for any $\fq \in \cS$, we have 
\cite[Theorem IV.10.4]{SilvermanAdvanced} that,
\[
	v_\fq(\fN) \; \le \; v_\fq(\cN) \; \le \;
	2+3 v_{\fq}(3) + 6 v_{\fq}(2) \; \le \; 2+6[L:\Q].
\]
Thus, the possible levels $\fN$ belong to a 
finite set that depends only on $E$ and $P$, and hence
the set of possible newforms $\ff$ depends only on $E$ and $P$.
Finally, by (b), $\fp$ is a prime of multiplicative
reduction for $F$ but $\ell \mid v_\fp (\Delta_F)$.
Thus, from the recipe \eqref{eqn:Nl},
we have that $\fp \mid\mid \cN$ but $\fp \nmid \fN$.
Moreover, as $\ell>\kappa^\prime \ge \Norm_{L/\Q}(\fp)$, we have that $\fp \nmid \ell$.
It follows from Remark~\ref{remark} that
\begin{equation}\label{eqn:congruence}
	\pm (\Norm_{L/\Q}(\fp)+1) \; \equiv \; a_{\fp}(\ff) \pmod{\lambda}.
\end{equation}
Since $\ell \mid \Norm_{\Q_{\ff}/\Q}(\lambda)$ we have that $\ell$ divides
\[
	\Norm_{\Q_{\ff}/\Q}\left( (\Norm_{L/\Q}(\fp)+1) \mp a_{\fp}(\ff) \right).
\]
This quantity is non-zero as $\lvert a_{\fp}(\ff) \rvert < 
2 \sqrt{\Norm_{L/\Q}(\fp)}$ (see for example
\cite[Theorem 0.1]{livne}). Hence $\ell$ is bounded by a bound
depending on $\ff$ and $\fp$, and hence ultimately
depending only on $E$ and $P$. This completes the proof.
\end{proof}

\section{Proof of Proposition~\ref{prop:prelim}}\label{sec:prop}

In this section, we complete the proof of Theorem~\ref{thm:Main1}
by proving Proposition~\ref{prop:prelim}.
We recall the Minkowski constant $M_L$ for a number field $L$:
\[
	M_L \; = \; \lvert \Delta_L \rvert^{1/2} \left(\frac{4}{\pi} \right)^{s} \frac{d!}{d^d}
\]
where $s$ is the number
of pairs of complex embeddings of $L$ and $d$ is its degree,
and $\Delta_L$ is its discriminant. 
The following lemma is a sharpening of an elementary result of 
Freitas and Siksek \cite[Lemma 3.2]{siksekgcd}.  
\begin{lemma}\label{lem:Minkowski}
Let $L$ be a number field. 
Let
\begin{equation}\label{eqn:Minkowski}
	\cT \; = \; \{ \fq \; : \; \text{$\fq$ is 
	a prime ideal of $\OO_L$ and $\Norm_{L/\Q}(\fq)<M_L$}\}.
\end{equation}	
Let $w_1,w_2, w_3 \in \OO_L \setminus \{0\}$. Then, there
exists $\alpha \in L^{\times}$ such that ${z_i}:=\alpha w_i \in
	\OO_L \setminus \{0\}$ and the ideal
${z_1}\OO_L+{z_2}\OO_L+{z_3}\OO_L$
is supported on $\cT$.  
\end{lemma}
We observe that $z_1 \OO_L+z_2 \OO_L+z_3\OO_L$ is the greatest common
divisor of the principal ideals $z_1 \OO_L$, $z_2 \OO_L$, $z_3 \OO_L$.
\begin{proof}
Write 
\begin{equation}\label{eqn:fb}
	\fb \;= \; w_1 \OO_L+w_2 \OO_L+w_3 \OO_L.
\end{equation}
We let $\Cl(L)$ denote the class group of $L$. 
We denote by $[\fa]$ the class in $\Cl(L)$
of an ideal $\fa$ of $\OO_L$.
Minkowski's Theorem asserts that the classes $[\fq]$
with $\fq \in \cT$ generate $\Cl(L)$.
Thus there is some (integral) ideal $\fa$
of $\OO_L$, supported on $\cT$,
such that $[\fb]=[\fa]$ and hence
$(\alpha \cdot \OO_L) \cdot 
	\fb  = 
	\fa$
for some $\alpha \in L^\times$.
Let $z_i=\alpha w_i$. Then 
\[
	z_1 \OO_L+z_2 \OO_L+z_3 \OO_L
	\; = \;
	(\alpha \cdot \OO_L) \cdot (w_1 \OO_L+w_2 \OO_L+w_3 \OO_L)
	\; = \; (\alpha \cdot \OO_L) \cdot \fb \; = \; \fa.
\]
It remains to show that the $z_i$ are integral. 
We note that
\[
		z_i \OO_L \; = \; (\alpha \OO_L) (w_i \OO_L)
		\; = \; \fa \cdot \fb^{-1}\cdot (w_i \OO_L) 
\]
which is an integral ideal, since $\fb \mid w_i$ by 
\eqref{eqn:fb}. This completes the proof.
\end{proof}

We quickly recall the notation and assumptions
of the previous section. We let 
$E/\Q$ be an elliptic curve given by an integral Weierstrass model
\eqref{eqn:model} and $P \in E(\Q)$. We suppose
that $E$ and $P$ satisfy conditions (i), (ii) and (iii) of
Theorem~\ref{thm:Main1}.
We let $A_n$, $B_n$, $C_n$ be given by \eqref{eqn:divseq}
and \eqref{eqn:cond}. We let $E^\prime$ be the short Weierstrass model
given in \eqref{eqn:Eprime} where the coefficients
$a$, $b$, $c$ are given in \eqref{eqn:sub1}, and we recall
the transformation \eqref{eqn:sub2} relating the two models.
We write $P^\prime \in E^\prime(\Q)$ for the point corresponding
to $P$, and we let $\theta_1$, $\theta_2$, $\theta_3$
be the roots of the polynomial ${x^\prime}^3+a {x^\prime}^2+ b x^\prime+c$,
and we define the number fields $K$ and $L$ by \eqref{eqn:KL}.
\begin{lemma}\label{lem:2descent}
For any positive integer $n$, and for $i=1$, $2$, $3$,
there exists some
$\varepsilon_{i,n} \in \OO_L \setminus \{0\}$
such that $4A_n-\theta_i B_n^2=\varepsilon_{i,n}^2$.
\end{lemma}
\begin{proof}
We make use of the standard $2$-descent homomorphism (e.g. 
\cite[Proposition X.1.4]{silverman}),
\[
	\phi:\ E^\prime(\mathbb{Q})/2E^\prime(\mathbb{Q})\longrightarrow \mathbb{Q}(\theta_1)^{\times}/(\mathbb{Q}(\theta_1)^\times)^2\times \mathbb{Q}(\theta_2)^{\times}/(\mathbb{Q}(\theta_2)^{\times})^2\times\mathbb{Q}(\theta_3)^{\times}/(\mathbb{Q}(\theta_3)^{\times})^2,
\]
which, for $(x^\prime,y^\prime) \in E^\prime(\Q) \setminus E[2]$, is given
by
\[
(x^\prime,y^\prime) \; \mapsto \; 
	\left((x^\prime-\theta_1)(\mathbb{Q}(\theta_1)^{\times})^2,(x^\prime-\theta_2)(\mathbb{Q}(\theta_2)^{\times})^2,(x^\prime-\theta_3)(\mathbb{Q}(\theta_3)^{\times})^2 \right).
\]
It follows that 
$x^\prime(nP^\prime)-\theta_i \in (K^\times)^2$
if $n$ is even, and 
$(x^\prime(nP^\prime)-\theta_i)/(x^\prime(P)-\theta_i) \in
(K^\times)^2$ if $n$ is odd.
From the definition of $L$ in \eqref{eqn:KL} we conclude that 
$x^\prime(nP^\prime)-\theta_i \in (L^\times)^2$.
The lemma follows as $x^\prime(nP^\prime)=4 A_n/B_n^2$
by \eqref{eqn:divseq} and \eqref{eqn:sub2}.
\end{proof}

\begin{lemma}\label{lem:gcd}
	Let $n$ be a positive integer. For any $i$, $j \in \{1,2,3\}$
	with $i \ne j$, we have that the ideal
\begin{equation}\label{eqn:gcd}
	(\varepsilon_{i,n}-\varepsilon_{j,n}) \OO_L+
	(\varepsilon_{i,n}+\varepsilon_{j,n}) \OO_L
\end{equation}
is supported on the prime ideals dividing $2 \Delta_E$,
where $\Delta_E$ is the discriminant of the model
 \eqref{eqn:model}.
\end{lemma}
\begin{proof}
Let $\fq \nmid 2$ be an ideal of $\OO_L$
dividing \eqref{eqn:gcd}. Then $\fq$
divides both $\varepsilon_{i,n}-\varepsilon_{j,n}$
and $\varepsilon_{i,n}+\varepsilon_{j,n}$,
and so both $\varepsilon_{i,n}$
and $\varepsilon_{j,n}$. Thus
\[
	4 A_n - \theta_i B_n^2  \; \equiv \;  4 A_n - \theta_j B_n^2 \; \equiv \; 0 
	\pmod{\fq}.
\]
Therefore, $\fq$ divides both $(\theta_i-\theta_j) B_n^2$
and $4(\theta_j-\theta_i) A_n$. Since $\gcd(A_n,B_n)=1$
we see that $\fq$ divides $4(\theta_j-\theta_i)$.
The lemma follows.
\end{proof}

\begin{proof}[Proof of Proposition~\ref{prop:prelim}]
We apply Lemma~\ref{lem:gcd} with the number field $L$
given by \eqref{eqn:KL}, and we let $\cT$
be as in \eqref{eqn:Minkowski}.
We let
\[
	\cS\; = \; \cT \cup \{ \fq \mid 2 \Delta_E \;
	: \;
	\text{$\fq$ is a prime ideal of $\OO_L$}\}.
\]
We let $T$ be the set of rational primes $q$
below the primes in $\cS$.
We now apply Lemma~\ref{lem:prime} with this particular
set $T$. We let $\kappa$ and $p \notin T$ be as in that lemma.
Let $n$, $u$ be positive integers and $\ell$ be a prime
such that $B_n=u^\ell$.
We suppose $\ell >\kappa$. Then, from the conclusion of Lemma~\ref{lem:prime},
we have $p \mid u$.
We henceforth fix a prime ideal $\fp$
of $\OO_L$ above $p$. We note in particular
that 
\begin{equation}\label{eqn:frakp}
\fp \notin \cS.
\end{equation}

From Lemma~\ref{lem:2descent} 
\begin{equation}\label{eqn:epsilon}
	4A_n-\theta_1 u^{2\ell} \; = \; \varepsilon_1^2,
	\qquad
	4A_n-\theta_2 u^{2\ell} \; =\; \varepsilon_2^2,
	\qquad
	4A_n-\theta_3 u^{2\ell} \; =\; \varepsilon_3^2,
\end{equation}
where $\varepsilon_i \in \OO_L\setminus \{0\}$ (we 
have simplified
the notation by writing $\varepsilon_i$ for $\varepsilon_{i,n}$).
We know that $\fp \mid p \mid u$.
Thus $\fp \mid (\varepsilon_i^2 - \varepsilon_j^2)$
for all $i$, $j \in \{1,2,3\}$,
and so $\fp \mid (\varepsilon_i-\varepsilon_j)$
or $\fp \mid (\varepsilon_i+\varepsilon_j)$. 
Replacing $\varepsilon_2$
by $-\varepsilon_2$ if necessary, we may suppose
$\fp \mid (\varepsilon_1-\varepsilon_2)$.
Replacing $\varepsilon_3$ by $-\varepsilon_3$
if necessary, we may suppose
that $\fp \mid (\varepsilon_2+\varepsilon_3)$.
We invoke Lemma~\ref{lem:gcd} which, in combination
with \eqref{eqn:frakp}, implies
that $\fp \nmid (\varepsilon_1+\varepsilon_2)$
and $\fp \nmid (\varepsilon_2-\varepsilon_3)$.
Since 
\[
	\varepsilon_3-\varepsilon_1 \; =\;
	-(\varepsilon_1-\varepsilon_2)-(\varepsilon_2-\varepsilon_3)
\]
we conclude that $\fp \nmid \varepsilon_3-\varepsilon_1$.
To summarize,
\begin{equation}\label{eqn:summary}
	\fp \mid (\varepsilon_1-\varepsilon_2),
	\qquad
	\fp \nmid (\varepsilon_2-\varepsilon_3),
	\qquad
	\fp \nmid (\varepsilon_3-\varepsilon_1).
\end{equation}
From \eqref{eqn:epsilon},
\begin{align*}
	(\varepsilon_1-\varepsilon_2)(\varepsilon_1+\varepsilon_2)=(\theta_2-\theta_1) u^{2 \ell} \\
	(\varepsilon_2-\varepsilon_3)(\varepsilon_2+\varepsilon_3)=(\theta_3-\theta_2) u^{2 \ell} \\
	(\varepsilon_3-\varepsilon_1)(\varepsilon_3+\varepsilon_1)=(\theta_1-\theta_3) u^{2 \ell} \\
\end{align*}
Since $u$ is a positive integer, we observe that $\varepsilon_i \ne \pm \varepsilon_j$ for $i \ne j$. 
Moreover, Lemma~\ref{lem:gcd} allows us to conclude that
\[
	\ell \mid v_\fq(\varepsilon_i-\varepsilon_j)  
	\qquad
	\text{for all $\fq \notin \cS$}.
\]
We let
\begin{equation}\label{eqn:wi}
	w_1=\varepsilon_1-\varepsilon_2,
	\qquad
	w_2=\varepsilon_2-\varepsilon_3,
	\qquad
	w_3=\varepsilon_3-\varepsilon_1.
\end{equation}
Thus
\[
	v_\fq(w_1) \equiv v_\fq(w_2) \equiv v_\fq(w_3) \equiv 0 \pmod{\ell},
	\qquad \text{for all $\fq \notin \cS$}.
\]
We shall make repeated use of the identity
\[
	w_1+w_2+w_3=0.
\]
By Lemma~\ref{lem:Minkowski},  there is some $\alpha \in L^\times$
such that $z_i:=\alpha w_i \in \OO_L \setminus \{0\}$ and 
\[
\fa:=z_1 \OO_L+z_2 \OO_L+z_3 \OO_L
\]
is supported on $\cT \subseteq \cS$. Let $\fq \notin \cS$.
Then $v_\fq(z_i)=0$ for at least one of $i=1$, $2$, $3$,
and so $v_\fq(\alpha)=v_\fq(z_i)-v_\fq(w_i) \equiv 0 \pmod{\ell}$.
It follows that 
\begin{equation}\label{eqn:zi1}
	v_\fq(z_1) \equiv v_\fq(z_2) \equiv v_\fq(z_3) \equiv 0 \pmod{\ell}, \qquad \text{for all $\fq \notin \cS$}.
\end{equation}
Next, we consider the valuations at $\fp$, which we recall does
not belong to $\cS$, by \eqref{eqn:frakp}. 
From \eqref{eqn:summary}
we have $v_\fp(z_2)=v_\fp(z_3)=v_\fp(\alpha)$. 
The integrality of the $z_i$ implies $v_\fp(\alpha) \ge 0$.
If $v_\fp(\alpha)>0$, then $\fp$ divides both $z_2$
and $z_3$ and so divides $z_1=-z_2-z_3$,
contradicting the fact that $\fa$ is supported on $\cS$.
Hence $v_\fp(\alpha)=0$. It follows from \eqref{eqn:summary} that
\begin{equation}\label{eqn:zi2}
	\fp \mid z_1, \qquad \fp \nmid z_2, \qquad \fp \nmid z_3.
\end{equation}

We let $F/L$ be the elliptic curve given by the integral
Weierstrass model
\[
	F \; : \; Y^2=X(X-z_1)(X+z_2).
\]
The discriminant and $c_4$-invariant of the model $F$ are
respectively
\[
	\Delta_F= 16 (z_1 z_2 z_3)^2, \qquad
c_4=16(z_1^2 - z_2 z_3)=16(z_2^2-z_3 z_1)=16(z_3^2-z_1z_2).
\]
The model is minimal and semistable away from the primes
dividing $2(z_1 \OO_L+z_2 \OO_L+z_3 \OO_L)=2\fa$
which all belong to $\cS$. Moreover, we note that
for $\fq \notin \cS$, we have $\ell \mid \ord_\fq(\Delta_F)$
from \eqref{eqn:zi1}. Finally, by \eqref{eqn:zi2},
the prime $\fp$ is a multiplicative prime for $F$.
This completes the proof.
\end{proof}

\section{Proof of Theorem~\ref{thm:Main2}}\label{sec:conjectures}

In this section we give precise statements of the three
conjectures mentioned in the introduction, and then we use
these to prove Theorem~\ref{thm:Main2}.
The main difference between Theorems~\ref{thm:Main1}
and~\ref{thm:Main2} is that in the latter the Frey curve $F$
is defined over a number field $L$ that is not totally real.
Our strategy follows closely arguments of 
\c{S}eng\"{u}n and Siksek \cite{SengunSiksek},
and of \c{T}urca\c{s} \cite{Turcas}, who proved versions 
of Fermat's Last Theorem over certain general number fields,
conditional on Serre's modularity conjecture
and Serre's uniformity conjecture.

The first conjecture we need 
is a weak version of Serre's modularity conjecture
over number fields. We recommend Sections 2 and 3 of \cite{SengunSiksek}
for an explanation of mod $\ell$ modular forms over number fields
and their Hecke operators, and further details of the 
terminology appearing in the statement of the conjecture.
\begin{conj} \label{conj:modularity}
Let $L$ be a number field and write $G_L=\Gal(\overline{\Q}/L)$.
Let $\ell \ge 5$ be a prime, and
let $\overline{\rho} : G_L \rightarrow \GL_2(\overline{\F}_\ell)$ be an odd,
irreducible, continuous representation with Serre conductor $\fN$ 
(prime-to-$\ell$
part of its Artin conductor) and trivial character (prime-to-$\ell$ part of
${\det}(\overline{\rho})$).
Assume that $\ell$ is unramified in $L$ and that $\overline{\rho}\vert_{G_{K_\fl}}$ arises from a finite-flat group scheme over the completion
$L_\fl$ for every prime $\fl \mid \ell$. Then there is a (weight $2$) mod $\ell$ eigenform $\theta$ over $L$ of level
$\fN$ such that for all prime ideals $\fq \nmid \ell\fN$, we have
\[
\theta(T_\fq)=
\Tr(\overline{\rho}
({\Frob}_\fq)). 
\]
\end{conj}

The second conjecture we need is widely known as Serre's uniformity
conjecture (see, for example \cite[Conjecture 1.3]{BELOV}).
\begin{conj}\label{conj:uniformity}
Let $L$ be a number field. Then there is a constant $C_L$, such that
for any non-CM elliptic curve $E/L$ and any prime $\ell>C_L$,
the mod $\ell$ representation $\overline{\rho}_{E,\ell}$
is surjective.
\end{conj}

The third conjecture we need is a special case
of what is widely known as the Ramanujan--Petersson conjecture.
A proof of this for CM fields has recently been
announced by Boxer, Calegari, Gee, Newton and Thorne 
\cite[Theorem A]{RamanujanPetersson}.
\begin{conj}\label{conj:RP}
Let $L$ be a number field.
Let $\ff$ be a weight $2$ complex eigenform over $L$
of level $\fN$, and let $\fp \nmid \fN$ be a prime ideal of $\OO_L$.
Then $\lvert a_\fp(\ff) \rvert \le 2 \sqrt{\Norm_{L/\Q}(\fp)}$.
\end{conj}

\begin{proof}[Proof of Theorem~\ref{thm:Main2}]
Most steps in the proof of Theorem~\ref{thm:Main1} carry on
unchanged to the setting of Theorem~\ref{thm:Main2}.
In particular, Proposition~\ref{prop:prelim}
continues to hold, with the only difference being that the number field $L$
given in \eqref{eqn:KL} may no longer be totally real.
We explain the derivation of Theorem~\ref{thm:Main2}
from Proposition~\ref{prop:prelim} assuming
the three aforementioned conjectures.

Again, we need to show that the exponent $\ell$ is bounded.
Let $F$, $\kappa$, $\cS$, $\fp$ be as in Proposition~\ref{prop:prelim}.	
There is currently no analogue of the level lowering
Theorem~\ref{thm:ll} in the setting of general number fields.
Instead we need to rely on Serre's modularity conjecture
(Conjecture~\ref{conj:modularity}) which combines the roles
of Theorems~\ref{thm:modularity} and~\ref{thm:ll},
and which we apply to the mod $\ell$ representation
$\overline{\rho}_{F,\ell}$.
Conjecture~\ref{conj:modularity} requires that
$\overline{\rho}_{F,\ell}$ is absolutely irreducible,
and not just irreducible. For this we rely on 
Serre's uniformity conjecture (Conjecture~\ref{conj:uniformity}),
and we replace the bound $C_L$ coming from Theorem~\ref{thm:irred}
with the bound $C_L$ coming from Conjecture~\ref{conj:uniformity};
for this we need that $F$ is non-CM, but this is true
as it has multiplicative reduction at $\fp$.
As before, $F$ is semistable at all $\fq \notin \cS$,
and satisfies $\ell \mid v_\fq(\Delta)$, and in particular
this is true for all $\fl \mid \ell$. This guarantees
(see for example \cite{Serre87}) that 
$\overline{\rho}_{F,\ell}$ is unramified
at $\fq \notin \cS$ satisfying $\fq \nmid \ell$,
and that it is finite flat at all $\fl \mid \ell$.
Denote the Serre conductor of $\overline{\rho}_{F,\ell}$
by $\fN^\prime$. This divides the \lq\lq level\rq\rq\
$\fN$ given by \eqref{eqn:Nl}, and indeed $\fN^\prime$ 
is the prime-to-$\ell$ factor of $\fN$. What matters is that
there are only finitely many possibilities for $\fN$ (depending
only on $E$ and $P$) and therefore only finitely many possibilities
for $\fN^\prime$. It follows from Conjecture~\ref{conj:modularity}
that there is a weight $2$  mod $\ell$ newform $\theta$
over $L$ such that for all $\fq \nmid \ell \fN$,
we have that $\theta(T_\fq)=\Tr(\overline{\rho}_{F,\ell}(\Frob_\fq))$.
Here we think of $\theta$ as a homomorphism from the Hecke algebra
into $\overline{\F}_\ell$, and $T_\fq$ denotes the
Hecke operator corresponding to $\fq$.
Our next step is to lift $\theta$ to a complex eigenform
$\ff$ without changing the level $\fN^\prime$.
We may do this for $\ell$ sufficiently large,
as explained in \cite[Section 2.1]{SengunSiksek},
building on a theorem of Ash and Stevens \cite[Section 1.2]{AshStevens}.
Finally we bound $\ell$ from the congruence \eqref{eqn:congruence}.
To do this we need the bound 
$\lvert a_{\fp}(\ff) \rvert < 2 \sqrt{\Norm_{L/\Q}(\fp)}$, which is supplied by Conjecture~\ref{conj:RP}.
\end{proof}

\section{Final Remarks}

\subsection*{Effectivity}
Theorem~\ref{thm:Main1} is not effective, as it relies on
the theorem of Everest, Reynolds and Stevens (Theorem~\ref{thm:ERS})
which in turn relies on Faltings' theorem.
It is however natural to ask if the bound for the
exponent $\ell$ yielded by the
proof of Theorem~\ref{thm:Main1} is effective.
Alas, it is not, because the finiteness of the 
set of non-modular $j$-invariants (Theorem~\ref{thm:modularity})
is a consequence of Faltings' theorem.
If we however assume modularity of elliptic curves over
$L$ then the bound on $\ell$ is effectively computable.
The constant $\kappa$ in Proposition~\ref{prop:prelim}
ultimately comes from Silverman's primitive divisor
theorem (Theorem~\ref{thm:silverman}) which has
been made effective \cite{Verzobio}. The paper
\cite{siksekbound} gives an effective
strategy for computing the constant $C_L$
in Theorem~\ref{thm:irred}. Finally, the Hilbert modular
forms appearing in the proof of Theorem~\ref{thm:Main1},
and their Hecke eigenvalues,
are algorithmically computable \cite{DembeleVoight}.

\subsection*{The non-integrality hypothesis}
The non-integrality hypothesis (assumption (i)
in Theorem~\ref{thm:Main1}) is essential to our argument,
as it had been in previous work of Reynolds \cite{Reynolds}
and Alfaraj \cite{Alfaraj2}. Indeed, this assumption
allows us to conclude the existence of a fixed 
multiplicative prime $\fp$ that is independent of $\ell$,
and this allowed us to bound $\ell$ in the final
step of the proof of Theorem~\ref{thm:Main1}.
If $P$ is integral, then $B_1=1=1^{\ell}$ and hence
$\ell$ is unbounded. It seems that the non-integrality 
hypothesis cannot be eliminated with some major new idea.

\bibliographystyle{plain}
\bibliography{ref}

\begin{thebibliography}{10}

\bibitem{Alfaraj2}
Abdulmuhsin Alfaraj.
\newblock On the finiteness of perfect powers in elliptic divisibility
  sequences.
\newblock {\em J. Th\'{e}or. Nombres Bordeaux}, 35(1):247--258, 2023.

\bibitem{AshStevens}
Avner Ash and Glenn Stevens.
\newblock Cohomology of arithmetic groups and congruences between systems of
  {H}ecke eigenvalues.
\newblock {\em J. Reine Angew. Math.}, 365:192--220, 1986.

\bibitem{BELOV}
Abbey Bourdon, \"{O}zlem Ejder, Yuan Liu, Frances Odumodu, and Bianca Viray.
\newblock On the level of modular curves that give rise to isolated
  {$j$}-invariants.
\newblock {\em Adv. Math.}, 357:106824, 33, 2019.

\bibitem{RamanujanPetersson}
George Boxer, Frank Calegari, Toby Gee, James Newton, and Jack~A. Thorne.
\newblock The ramanujan and sato-tate conjectures for bianchi modular forms,
  2023.
\newblock \url{https://arxiv.org/abs/2309.15880}.

\bibitem{FibLuc}
Yann Bugeaud, Maurice Mignotte, and Samir Siksek.
\newblock Classical and modular approaches to exponential {D}iophantine
  equations. {I}. {F}ibonacci and {L}ucas perfect powers.
\newblock {\em Ann. of Math. (2)}, 163(3):969--1018, 2006.

\bibitem{CornelissenReynolds}
Gunther Cornelissen and Jonathan Reynolds.
\newblock The perfect power problem for elliptic curves over function fields.
\newblock {\em New York J. Math.}, 22:95--114, 2016.

\bibitem{SengunSiksek}
Mehmet~Haluk \c{S}eng\"{u}n and Samir Siksek.
\newblock On the asymptotic {F}ermat's last theorem over number fields.
\newblock {\em Comment. Math. Helv.}, 93(2):359--375, 2018.

\bibitem{Turcas}
George~C. \c{T}urca\c{s}.
\newblock On {S}erre's modularity conjecture and {F}ermat's equation over
  quadratic imaginary fields of class number one.
\newblock {\em J. Number Theory}, 209:516--530, 2020.

\bibitem{DembeleVoight}
Lassina Demb\'{e}l\'{e} and John Voight.
\newblock Explicit methods for {H}ilbert modular forms.
\newblock In {\em Elliptic curves, {H}ilbert modular forms and {G}alois
  deformations}, Adv. Courses Math. CRM Barcelona, pages 135--198.
  Birkh\"{a}user/Springer, Basel, 2013.

\bibitem{Everest}
Graham Everest, Jonathan Reynolds, and Shaun Stevens.
\newblock On the denominators of rational points on elliptic curves.
\newblock {\em Bull. Lond. Math. Soc.}, 39(5):762--770, 2007.

\bibitem{SamiretalModular}
Nuno Freitas, Bao~V. Le~Hung, and Samir Siksek.
\newblock Elliptic curves over real quadratic fields are modular.
\newblock {\em Invent. Math.}, 201(1):159--206, 2015.

\bibitem{SamirSummary}
Nuno Freitas and Samir Siksek.
\newblock The asymptotic {F}ermat's last theorem for five-sixths of real
  quadratic fields.
\newblock {\em Compos. Math.}, 151(8):1395--1415, 2015.

\bibitem{siksekgcd}
Nuno Freitas and Samir Siksek.
\newblock The asymptotic {F}ermat's last theorem for five-sixths of real
  quadratic fields.
\newblock {\em Compos. Math.}, 151(8):1395--1415, 2015.

\bibitem{siksekbound}
Nuno Freitas and Samir Siksek.
\newblock Criteria for irreducibility of {${\rm mod}\, p$} representations of
  {F}rey curves.
\newblock {\em J. Th\'{e}or. Nombres Bordeaux}, 27(1):67--76, 2015.

\bibitem{Hajdu}
Lajos Hajdu, Shanta Laishram, and M\'{a}rton Szikszai.
\newblock Perfect powers in products of terms of elliptic divisibility
  sequences.
\newblock {\em Bull. Aust. Math. Soc.}, 94(3):395--404, 2016.

\bibitem{livne}
R.~Livn\'{e}.
\newblock Communication networks and {H}ilbert modular forms.
\newblock In {\em Applications of algebraic geometry to coding theory, physics
  and computation ({E}ilat, 2001)}, volume~36 of {\em NATO Sci. Ser. II Math.
  Phys. Chem.}, pages 255--270. Kluwer Acad. Publ., Dordrecht, 2001.

\bibitem{Mazur78}
B.~Mazur.
\newblock Rational isogenies of prime degree (with an appendix by {D}.
  {G}oldfeld).
\newblock {\em Invent. Math.}, 44(2):129--162, 1978.

\bibitem{Merel}
Lo\"{\i}c Merel.
\newblock Bornes pour la torsion des courbes elliptiques sur les corps de
  nombres.
\newblock {\em Invent. Math.}, 124(1-3):437--449, 1996.

\bibitem{EkinSurvey}
Ekin \"{O}zman and Samir Siksek.
\newblock {$S$}-unit equations and the asymptotic {F}ermat conjecture over
  number fields.
\newblock In {\em Number theory}, De Gruyter Proc. Math., pages 83--103. De
  Gruyter, Berlin, [2022] \copyright 2022.

\bibitem{Petho}
Attila Peth\H{o}.
\newblock Perfect powers in second order linear recurrences.
\newblock {\em J. Number Theory}, 15(1):5--13, 1982.

\bibitem{Reynolds}
Jonathan Reynolds.
\newblock Perfect powers in elliptic divisibility sequences.
\newblock {\em J. Number Theory}, 132(5):998--1015, 2012.

\bibitem{Ribet-1990}
K.~A. Ribet.
\newblock On modular representations of {${\rm Gal}(\overline{\bf Q}/{\bf Q})$}
  arising from modular forms.
\newblock {\em Invent. Math.}, 100(2):431--476, 1990.

\bibitem{Serre87}
Jean-Pierre Serre.
\newblock Sur les repr\'{e}sentations modulaires de degr\'{e} {$2$} de
  {$\Gal(\overline{\mathbf{Q}}/\mathbf{Q})$}.
\newblock {\em Duke Math. J.}, 54(1):179--230, 1987.

\bibitem{ShoreyStewart}
T.~N. Shorey and C.~L. Stewart.
\newblock On the {D}iophantine equation {$ax\sp{2t}+bx\sp{t}y+cy\sp{2}=d$} and
  pure powers in recurrence sequences.
\newblock {\em Math. Scand.}, 52(1):24--36, 1983.

\bibitem{SiksekModular}
Samir Siksek.
\newblock The modular approach to {D}iophantine equations.
\newblock In {\em Explicit methods in number theory}, volume~36 of {\em Panor.
  Synth\`eses}, pages 151--179. Soc. Math. France, Paris, 2012.

\bibitem{SILVERMANABC}
Joseph~H. Silverman.
\newblock Wieferich's criterion and the {$abc$}-conjecture.
\newblock {\em J. Number Theory}, 30(2):226--237, 1988.

\bibitem{SilvermanAdvanced}
Joseph~H. Silverman.
\newblock {\em Advanced topics in the arithmetic of elliptic curves}, volume
  151 of {\em Graduate Texts in Mathematics}.
\newblock Springer-Verlag, New York, 1994.

\bibitem{silverman}
Joseph~H. Silverman.
\newblock {\em The arithmetic of elliptic curves}, volume 106 of {\em Graduate
  Texts in Mathematics}.
\newblock Springer, Dordrecht, second edition, 2009.

\bibitem{Streng}
Marco Streng.
\newblock Elliptic divisibility sequences with complex multiplication.
\newblock Master's thesis, Universiteit Utrecht, 2006.
\newblock Available at
  \url{https://www.math.leidenuniv.nl/~streng/mthesis.pdf}.

\bibitem{Verzobio}
Matteo Verzobio.
\newblock Some effectivity results for primitive divisors of elliptic
  divisibility sequences.
\newblock {\em Pacific J. Math.}, 325(2):331--351, 2023.

\bibitem{Wiles}
Andrew Wiles.
\newblock Modular elliptic curves and {F}ermat's last theorem.
\newblock {\em Ann. of Math. (2)}, 141(3):443--551, 1995.

\end{thebibliography}
\end{document}